\documentclass[12pt, a4paper]{article}

\usepackage{amssymb}
\usepackage{amsmath}
\usepackage{amsthm}

  \newtheorem{theorem}{Theorem}

    \newtheorem{lemma}{Lemma}    
    \newtheorem{proposition}{Proposition} 
    
      \newcommand{\dy}{\,\mathrm{d}y}

\newcommand{\ds}{\,\mathrm{d}s}
\newcommand{\dt}{\,\mathrm{d}t}
\newcommand{\dsig}{\,\mathrm{d}\sigma}

    \newcommand\vN{{\bf  \nabla}}
    \newcommand\ddh{--}
    \DeclareMathOperator{\Div}{div}
    
\begin{document}

\title{A short note on maximal regularity for the heat kernel in Besov spaces and a degenerate 3D Keller\ddh Segel system.}
\author{  Pierre Gilles Lemari\'e-Rieusset\footnote{LaMME, Univ Evry, CNRS, Universit\'e Paris-Saclay, 91025, Evry, France}\footnote{ emeritus professor, (Univ Evry)} \footnote{e-mail : pierregilles.lemarierieusset@univ-evry.fr}}
\date{}\maketitle

\begin{abstract} We prove global existence  for a degenerate Keller\ddh Segel equation with small initial values in a critical Lorentz space. 
\end{abstract}
    
    We discuss the following partial differential system  on $(0,+\infty)\times\mathbb{R}^3$
    \begin{equation}\label{Keller1}\left\{ \begin{split} \partial_t u  =&\Delta u-\Div (u\vN \varphi)\\
    \partial_t \varphi=&\epsilon \Delta\varphi +u\\ u(0,.)=& u_0
    \\ \varphi(0,.)=&0\end{split}\right.\end{equation} where $\epsilon\geq 0$. For $\epsilon>0$, this is called the parabolic-parabolic Keller\ddh Segel system. The case $\epsilon=0$ can be seen as a degeneracy of  the parabolic-parabolic Keller\ddh Segel system.
    
    Using Duhamel's formula for the heat equation, (\ref{Keller1}) is rewritten as    \begin{equation}\label{Keller2}\left\{ \begin{split}   u  =&e^{t\Delta} u_0-\int_0^t e^{(t-s)\Delta}\Div (u\vN \varphi) \ds \\
      \varphi=  &\int_0^t   e^{\epsilon(t-s)\Delta} u\ds \end{split}\right.\end{equation}  and finally we get that
      \begin{equation}\label{Keller3}
      u= e^{t\Delta} u_0-\int_0^t e^{(t-s)\Delta}\Div \left(u \int_0^s   e^{\epsilon(s-\sigma)\Delta} \vN u\dsig \right) \ds . \end{equation}
      
      As usual, mild solutions are described of fixed points of the transform
      \[ u\mapsto e^{t\Delta}u_0-B_\epsilon(u,u)\] with \[ B_\epsilon(u,v)= \int_0^t e^{(t-s)\Delta}\Div \left(u \int_0^s   e^{\epsilon(s-\sigma)\Delta} \vN v\dsig \right) \ds .\] When $\epsilon=0$, $B_0$ is given by \[ B_0(u,v)= \int_0^t e^{(t-s)\Delta}\Div \left(u \int_0^s    \vN v\dsig \right) \ds .\] 
      The idea for solving (\ref{Keller3}) is then to identify a Banach space $\mathbb{X}$ of distributions on $\mathbb{R}^3$ and a Banach space $ \mathbb{Y}$ of functions on $(0,+\infty)\times\mathbb{R}^3$ such that:
      \begin{itemize} \item[$\bullet$] for $u_0\in \mathbb{X}$, $e^{t\Delta}u_0\in \mathbb{Y}$ and $\| e^{t\Delta}u_0\|_{\mathbb{Y}} \leq C_{\mathbb{X},\mathbb{Y}} \|u_0\|_{\mathbb{X}}$.
  \item[$\bullet$]
for $(u, v)\in\mathbb{Y}\times\mathbb{Y}$, $B_\epsilon(u,v)\in \mathbb{Y}$ and $\| B_\epsilon(u,v)\|_{\mathbb{Y}} \leq C_{\epsilon,\mathbb{Y}} \|u\|_{\mathbb{Y}}  \|v\|_{\mathbb{Y}}$. 
      \end{itemize}  Then, for small $u_0\in\mathbb{X}$, i. e. for $u_0$ such that
\[ \|u_0\|_\mathbb{X}<\frac 1{ 4 C_{\mathbb{X},\mathbb{Y}}C_{\epsilon,\mathbb{Y}}} ,\] the Picard iterates $w_n$ defined by
\[ w_0=0\text{ and } w_{n+1}= e^{t\Delta}u_0-B_\epsilon(w_n,w_n)\] converge in $\mathbb{Y}$ to a solution $u$ such that \[ \|u\|_{\mathbb{Y}} <\frac 1{ 2  C_{\epsilon,\mathbb{Y}}}   .\]

    If $(u,\varphi)$ is a solution of (\ref{Keller1})  with initial value $u_0$ and if $\lambda>0$ and $x_0\in\mathbb{R}^3$, then, defining  \[u_\lambda(t,x)=\lambda^2 u(\lambda^2 t,\lambda(x-x_0)) \text{ and } \varphi_\lambda(t,x)=\varphi(\lambda^2 t,\lambda(x-x_0)) ,\] we easily check that $(u_\lambda,\varphi_\lambda)$ is a solution of the system (\ref{Keller1})  with initial value $\lambda^2 u_0(\lambda (x-x_0))$. Looking for global solutions, one usually deals with small initial values in a Banach space $\mathbb{X}$ which satisfies the criticality condition 
    \[ \lambda^2 \| u_0(\lambda (x-x_0))\|_\mathbb{X}=\|u_0\|_\mathbb{X}.\] 
    
    The case $\epsilon>0$ has been studied in many such frames, for instance Lebesgue  space $\mathbb{X}=L^{\frac 3 2}$ (Calvez, Corrias and Ebde \cite{Calv12}), Lorentz space $\mathbb{X}=L^{\frac 3 2,\infty}$ (Kozono and Sugiyama \cite{Kozo08}), Besov space $\mathbb{X}=\dot B^{-\frac 1 2}_{2,\infty}$ (Iwabuchi \cite{Iwab11}), Morrey space $\mathbb{X}=\dot M^{1,\frac 3 2}$ (Lemari\'e-Rieusset \cite{Lema13}).  [Recall that, for $1\leq q\leq +\infty$, $1\leq r\leq +\infty$ and $\gamma\in\mathbb{R}$ with $\gamma<\frac 3 q$ if $r>1$ or $\gamma\leq \frac 3 q$ if $r=1$, the homogeneous Besov space $\dot B^\gamma_{q,r}(\mathbb{R}^3)$ is characterized by the Littlewood\ddh Paley decomposition
    \[ u\in \dot B^\gamma_{q,r} \iff u=\sum_{j\in\mathbb{Z}}\Delta_j u \text{ in }\mathcal{S}'\text{ and } (2^{j\gamma} \|\Delta_ju\|_q)_{j\in\mathbb{Z}}\in l^r,\]   (see Lemari\'e-Rieusset \cite{Lema02} for instance).]
    
    In section \ref{Sect1}, we recall those results in a simple way:
    
    \begin{proposition} \label{prop1}  Let $\epsilon>0$ and $\frac 1 4<\alpha<\frac 1 2$.  Consider the equation  
      \begin{equation}\label{Keller4}
      u= e^{t\Delta} u_0-\int_0^t e^{(t-s)\Delta}\Div \left(u \int_0^s   e^{\epsilon(s-\sigma)\Delta} \vN u\dsig \right) \ds . \end{equation}There exists a constant $C_{\alpha}$ such that: \begin{itemize}
      \item[$\bullet$] if $u_0\in L^{3/2}$ and $\|u_0\|_{3/2} <C_{\alpha} \sqrt\epsilon$ , then (\ref{Keller4}) has a global solution $u$ such that \[u\in L^\infty((0,+\infty), L^{3/2}) \text{ and }\sup_{t>0} t^{1-\alpha} \|u(t,.)\|_{\frac 3{2\alpha}}<+\infty \]
      \item[$\bullet$]  if $u_0\in L^{3/2,\infty}$ and $\|u_0\|_{L^{3/2,\infty}} <C_{\alpha} \sqrt\epsilon$, then (\ref{Keller4}) has a global solution $u$ such that \[u\in L^\infty((0,+\infty), L^{3/2,\infty}) \text{ and }\sup_{t>0} t^{1-\alpha} \|u(t,.)\|_{L^{\frac3{2\alpha},\infty}}<+\infty\]
      \item[$\bullet$]  if $u_0\in \dot B^{-\frac 1 2}_{2,\infty}$  and $\|u_0\|_{ \dot B^{-\frac 1 2}_{2,\infty}} <C_{\alpha} \sqrt\epsilon$, then (\ref{Keller4}) has a global solution $u$ such that \[u\in L^\infty((0,+\infty), \dot B^{-\frac 1 2}_{2,\infty}) \text{ and }\sup_{t>0}t^{1-\alpha}   \|u(t,.)\|_{\dot B^{\frac 3 2-2\alpha}_{2,\infty}}<+\infty.\]
      \end{itemize}
    \end{proposition}
    
    Proposition \ref{prop1} shows the role of the relaxation parameter $\epsilon$: the larger it is, the easier it is to get a global solution. However, we shall be interested in smaller $\epsilon$'s, and even in the case $\epsilon=0$. (The case $\epsilon=0$ was a question I was asked by my colleague Nicolas Meunier.) Our main result is the following one (proved in Section \ref{Sect2}):
    
    \begin{theorem} \label{theo1}  Let $\epsilon\geq 0$.  Consider the equation  
      \begin{equation}\label{Keller5}
      u= e^{t\Delta} u_0-\int_0^t e^{(t-s)\Delta}\Div \left(u \int_0^s   e^{\epsilon(s-\sigma)\Delta} \vN u\dsig \right) \ds . \end{equation}There exists a positive constant $C_0$ (which doesn't depend on $\epsilon$) such that,  if $u_0\in L^{3/2,1}$ and $\|u_0\|_{L^{3/2,1}} <C_0 $ , then (\ref{Keller5}) has a global solution $u$ such that \[u\in L^\infty((0,+\infty), L^{3/2,\infty})  \cap L^1((0,+\infty), \dot B^{3/2}_{2,1}). \]       \end{theorem}
    
    \section{The case $\epsilon>0$}\label{Sect1}
    We now prove Proposition \ref{prop1}:
    \subsection*{Estimate on the heat kernel:}  The operator $e^{t\Delta}$ is given as the convolution with the heat kernel \[W_t(x)=\frac 1{(4\pi t)^{3/2}}e^{-\frac{\vert x\vert^2}{4t}} =\mathcal{F}^{-1}(e^{-t\vert\xi\vert^2})\]
    For $0<\beta<2$, the operator $(-\Delta)^{\beta/2}e^{t\Delta}$ is given as \[(-\Delta)^{\beta/2}e^{t\Delta}=\frac 1{\Gamma(1-\frac\beta 2)}\int_0^{+\infty}(-\Delta) e^{(t+s)\Delta}\frac{\ds}{s^{\frac \beta 2}} . \] It is a convolution with a kernel $K_{t,\beta}$ for which we have the estimate
    \[ \| K_{t,\beta}\|_1\leq \frac 1{\Gamma(1-\frac\beta 2)}\int_0^{+\infty} \|\Delta W_{t+s}\|_1 \frac{\ds}{s^{\frac \beta 2}} =\frac{\|\Delta W_1\|_1}{\Gamma(1-\frac\beta 2)}\int_0^{+\infty} \frac{\ds}{(t+s) s^{\beta/2}}\]
    and thus
     \[ \| K_{t,\beta}\|_1\leq  t^{-\frac \beta 2} \frac{\|\Delta W_1\|_1}{\Gamma(1-\frac\beta 2)}\int_0^{+\infty} \frac{\dsig}{(1+\sigma) \sigma^{\beta/2}}.\]
     
     In particular, if $0<\beta<2$ and   if $X_0$ and $X_\beta$ are two Banach spaces of distributions on $\mathbb{R}^3$ which are stable for convolution with $L^1$ functions ($\|f*u\|_{X_0}\leq \|f\|_1\|u\|_{X_0}$ and $\|f*u\|_{X _\beta}\leq \|f\|_1\|u\|_{X_\beta}$) and if $(-\Delta)^{-\beta/2}$ is a bounded operator from $X_0$ to $X_\beta$ ($\| \int \frac 1{\vert x-y\vert^{3-\beta}} u(y)  \dy\|_{X_\beta}\leq C_\beta \|u\|_{X_0}$), then we have the inequalities
     \begin{equation}  \|e^{t\Delta}u\|_{X_0}\leq \|u\|_{X_0 }\end{equation}and \begin{equation} \label{ineq1}\|e^{t\Delta}u\|_{X_\beta} \leq \|K_{t,\beta}\|_1 \|(-\Delta)^{-\beta/2} u\|_{X_\beta}\leq C_\beta t^{-\beta/2} \|u\|_{X_0}.\end{equation}
    \subsection*{Estimates on $e^{t\Delta} u_0$:}  
    Taking $X_0= L^{3/2}$, $\beta=2-2\alpha$ and $X_\beta=L^{\frac 3{2\alpha}}$, we get that
    \[ \sup_{t>0} \|e^{t\Delta}u_0\|_{\frac 3 2} +\sup_{t>0} t^{1-\alpha}\|e^{t\Delta}u_0\|_{\frac 3 {2\alpha}} \leq C_\alpha \|u_0\|_{\frac 3 2}.\]
     Similarly, taking  $X_0= L^{3/2,\infty}$, $\beta=2-2\alpha$ and $X_\beta=L^{\frac 3{2\alpha},\infty}$, we get that
    \[ \sup_{t>0} \|e^{t\Delta}u_0\|_{L^{3/2,\infty}} +\sup_{t>0} t^{1-\alpha}\|e^{t\Delta}u_0\|_{L^{\frac 3{2\alpha},\infty}} \leq C_\alpha \|u_0\|_{L^{3/2,\infty}}.\]
    Finally,   taking  $X_0=  \dot B^{-\frac 1 2}_{2,\infty}$, $\beta=2-2\alpha$ and $X_\beta= \dot B^{\frac 3 2-2\alpha}_{2,\infty}$, we get that
    \[ \sup_{t>0} \|e^{t\Delta}u_0\|_{ \dot B^{-\frac 1 2}_{2,\infty}} +\sup_{t>0} t^{1-\alpha}\|e^{t\Delta}u_0\|_{ \dot B^{\frac 3 2-2\alpha}_{2,\infty}} \leq C_\alpha \|u_0\|_{ \dot B^{-\frac 1 2}_{2,\infty}}.\]

    \subsection*{Estimate on $B_\epsilon(u,v)$ (case of Lebesgue space $L^{3/2}$):}  
    Let $\mathbb{Y}=\{u\in L^\infty((0,+\infty), L^{3/2})\ /\ \sup_{t>0} t^{1-\alpha} \|u(t,.)\|_{L^{\frac 3{2\alpha}}}<+\infty\}$ endowed with the norm
    \[ \|u\|_\mathbb{Y}=\sup_{t>0} \|u(t,.)\|_{3/2}+ \sup_{t>0} t^{1-\alpha} \|u(t,.)\|_{L^{\frac 3{2\alpha}}}.\]
    Let $u,v\in\mathbb{Y}$ and $w= B_\epsilon(u,v)= \int_0^t e^{(t-s)\Delta}\Div \left(u \int_0^s   e^{\epsilon(s-\sigma)\Delta} \vN v\dsig \right) \ds$. We have
    \[ \| \int_0^s   e^{\epsilon(s-\sigma)\Delta} \vN v\dsig \|_{\frac 3{2\alpha}}\leq C\int_0^s \frac 1{\sqrt{\epsilon (s-\sigma)}} \frac{\dsig}{\sigma^{1-\alpha}} \|v\|_{\mathbb{Y}}= C_\alpha\frac 1{ \sqrt\epsilon s^{\frac 1 2-\alpha}} \|v\|_\mathbb{Y}.\]
    We write 
    \[ \| B_\epsilon(u,v)(t,.)\|_{3/2}\leq C \int_0^t \frac 1{\sqrt{t-s}}\|e^{\frac{(t-s)}2 \Delta}(u \int_0^s   e^{\epsilon(s-\sigma)\Delta} \vN v\dsig )\|_{3/2}\ds\]
    and taking $X_0= L^{\frac 3{2+2\alpha}}$, $\beta=2\alpha $ and $X_\beta=L^{3/2}$, we get by (\ref{ineq1})
    \[ \|e^{\frac{(t-s)}2 \Delta}(u \int_0^s   e^{\epsilon(s-\sigma)\Delta} \vN v\dsig )\|_{3/2} \leq C_\alpha (t-s)^{-\alpha} \|u(s,.)\|_{3/2} \| \int_0^s   e^{\epsilon(s-\sigma)\Delta} \vN v\dsig \|_{\frac 3{2\alpha}} \] and thus
     \[ \| B_\epsilon(u,v)(t,.)\|_{3/2}\leq C_\alpha \frac 1{\sqrt\epsilon}\|u\|_\mathbb{Y} \|v\|_\mathbb{Y} \int_0^t \frac {\ds}{(t-s)^{\frac 1 2 +\alpha} s^{\frac 1 2-\alpha} }=C'_\alpha \frac 1{\sqrt\epsilon}\|u\|_\mathbb{Y} \|v\|_\mathbb{Y} . \]   Similarly, we write 
    \[ \| B_\epsilon(u,v)(t,.)\|_{\frac 3{2\alpha}}\leq C \int_0^t \frac 1{\sqrt{t-s}}\|e^{\frac{(t-s)}2 \Delta}(u \int_0^s   e^{\epsilon(s-\sigma)\Delta} \vN v\dsig )\|_{\frac 3{2\alpha}}\ds\]
    and taking $X_0= L^{\frac 3{4\alpha}}$, $\beta=2\alpha $ and $X_\beta=L^{\frac 3{2\alpha}} $, we get by (\ref{ineq1})
    \[ \|e^{\frac{(t-s)}2 \Delta}(u \int_0^s   e^{\epsilon(s-\sigma)\Delta} \vN v\dsig )\|_{\frac 3{2\alpha}}  \leq C_\alpha (t-s)^{-\alpha} \|u(s,.)\|_{\frac 3{2\alpha}}  \| \int_0^s   e^{\epsilon(s-\sigma)\Delta} \vN v\dsig \|_{\frac 3{2\alpha}} \] and thus (since $\int_0^1  \frac {\dsig}{(1-\sigma)^{\frac 1 2 +\alpha} \sigma^{\frac 3 2-2\alpha} }<+\infty$ for $1/4<\alpha<1/2$)
     \[ \| B_\epsilon(u,v)(t,.)\|_{\frac 3{2\alpha}} \leq C_\alpha \frac 1{\sqrt\epsilon}\|u\|_\mathbb{Y} \|v\|_\mathbb{Y} \int_0^t \frac {\ds}{(t-s)^{\frac 1 2 +\alpha} s^{\frac 3 2-2\alpha} }=C'_\alpha t^{\alpha-1} \frac 1{\sqrt\epsilon}\|u\|_\mathbb{Y} \|v\|_\mathbb{Y} . \]
     We conclude that $\| B_\epsilon(u,v)\|_\mathbb{X}\leq C_\alpha \frac 1{ \sqrt\epsilon} \|u\|_{\mathbb{X}}\|v\|_\mathbb{X}$ and Proposition 1 is proved in the case $u_0\in L^{3/2}$.
     
    \subsection*{Estimate on $B_\epsilon(u,v)$ (case of Lorentz space $L^{3/2,\infty}$):}  
    The proof for $u_0\in L^{3/2}$ can easily be adapted to the case $u_0\in L^{3/2,\infty}$, replacing $L^{3/2}$ with $L^{3/2,\infty}$, $L^{\frac 3{2\alpha}}$ with $L^{\frac 3{2\alpha},\infty}$, $L^{\frac 3{2+2\alpha}} $ with $L^{\frac 3{2+2\alpha},\infty}$, and $L^{\frac 3{4\alpha}}$ with $L^{\frac 3{4\alpha},\infty}$.
     
    \subsection*{Estimate on $B_\epsilon(u,v)$ (case of Besov space $ \dot B^{-\frac 1 2}_{2,\infty}$):}  
    The proof for $u_0\in L^{3/2}$ can easily be adapted to the case $u_0\in  \dot B^{-\frac 1 2}_{2,\infty}$, replacing $L^{3/2}$ with $ \dot B^{-\frac 1 2}_{2,\infty}$, $L^{\frac 3{2\alpha}}$ with $ \dot B^{\frac 3 2-2\alpha}_{2,\infty}$, $L^{\frac 3{2+2\alpha}} $ with $ \dot B^{-\frac 1 2-2\alpha}_{2,\infty}$, and $L^{\frac 3{4\alpha}}$ with $ \dot B^{\frac 3 2-4\alpha}_{2,\infty}$.
    
    The key ingredient for dealing with the non-linearity $u \int_0^t e^{\epsilon(t-s)\Delta}\vN v\ds$ is the product law in Besov spaces: for $\gamma<3/2$, $\delta<3/2$ and $\gamma+\delta>0$, we have
    \begin{equation}\label{product1} \|uw\|_{\dot B^{\gamma+\delta-\frac 3 2}_{2,\infty}} \leq C \|u\|_{\dot B^\gamma_{2,\infty} }\|w\|_{\dot  B^\delta_{2,\infty} } .\end{equation}
    
    \section{The case $\epsilon\geq 0$}\label{Sect2}
    
The case $\epsilon=0$ will be dealt with by using limit cases of the maximal regularity of the heat kernel.

  \section*{Maximal regularity of the heat kernel} 
We consider the operator $u\mapsto \mathcal{L}(u)$ defined by
\[ \mathcal{L}(u)(t,.)=\int_0^t   e^{(t-s)\Delta} \Delta u\ds.\]  The $L^pL^q$ maximal regularity of the heat kernel states that  $\mathcal{L}$ is bounded  on $L^p(0,+\infty[, L^q(\mathbb{R}^3))$ for $1<p<+\infty$ and $1<q<+\infty$. In the limit case $p=1$, one replaces $L^q$ with the  Besov space $\dot B^0_{q,1} $, and in the  limit case  $p=+\infty$ one replaces $L^q$ with $\dot B^0_{q,\infty}$.
\begin{lemma} For $1<q<+\infty$, we have the inequalities
\[ \int_0^{+\infty} \| \mathcal{L}(u)(t,.)\|_{\dot B^0_{q,1} }  \dt \leq C_q  \int_0^{+\infty} \| u(t,.)\|_{\dot B^0_{q,1} } \dt \] and 
\[ \sup_{0<t<+\infty} \|\mathcal{L}(u)(t.)\|_{\dot B^0_{q,\infty} }  \leq C_q \sup_{0<t<+\infty}  \| u(t,.)\|_{\dot B^0_{q,1\infty }} . \] 
\end{lemma}

\begin{proof} Using the Littlewood\ddh Paley decomposition, we see that
\[ \int_0^{+\infty} \| \mathcal{L}(u)(t,.)\|_{\dot B^0_{q,1} }   \dt =\int_0^{+\infty} \sum_{j\in \mathbb{Z}} \| \Delta_j\mathcal{L}(u)(t,.)\|_q \dt = \sum_{j\in \mathbb{Z}}   \int_0^{+\infty} \|\mathcal{L}(\Delta_j u)(t,.)\|_q \dt \] 
and \[\sup_{0<t<+\infty}  \| \mathcal{L}(u)(t,.)\|_{\dot B^0_{q,\infty} } = \sup_{0<t<+\infty}  \sup_{j\in \mathbb{Z}} \|\Delta_j\mathcal{L}(u)(t,.)\|_q  = \sup_{j\in \mathbb{Z}} \sup_{0<t<+\infty}  \|\mathcal{L}(\Delta_j u)(t,.)\|_q .\]  
We then check that 
\[ \int_0^{+\infty} \|\mathcal{L}(\Delta_j u)(t,.)\|_q \dt \leq C_q  \int_0^{+\infty} \| \Delta_j u(t,.)\|_q \dt  \] and 
\[ \sup_{0<t<+\infty}   \|\mathcal{L}(\Delta_j u)(t,.)\|_q \leq C_q  \sup_{0<t<+\infty}  \|\Delta_j u(t,.)\|_q .\]    We write 
\[   \|\mathcal{L}(\Delta_j u)(t,.)\|_q \leq \int_0^t  \|e^{(t-s)\Delta}\Delta\Delta_j u(s,.)\|_q\ds\leq C \int_0^t \min(\frac{2^{j}}{\sqrt{t-s}}, \frac {2^{-j}}{(t-s)^{3/2}}) \|\Delta_ju(s,.)\|_q\, ds\]
We conclude by noticing that
\[ \int_s^{+\infty} \min(\frac{2^{j}}{\sqrt{t-s}}, \frac {2^{-j}}{(t-s)^{3/2}}) \,dt=\int_0^{+\infty} \min(\frac 1{\sqrt t}, \frac 1{t^{3/2}})\dt<+\infty \]
and 
\[ \int_{-\infty}^t \min(\frac{2^{j}}{\sqrt{t-s}}, \frac {2^{-j}}{(t-s)^{3/2}}) \,ds=\int_0^{+\infty} \min(\frac 1{\sqrt s}, \frac 1{s^{3/2}})\ds<+\infty.\tag*{\qedhere} \]
\end{proof}
 
\noindent{\bf Remark:}  The regularity index $\sigma=0$ in $\dot B^{\sigma}_{q,r}$  can easily be replaced by    $\sigma\leq\frac 3 q $ ($r=1$) or $\sigma<\frac 3 q $ ($r=+\infty$), since
\[ \|\mathcal{L}(u)\|_{\dot B^{\sigma}_{q,r}}= \|(-\Delta)^{\frac\sigma 2}\mathcal{L}(u)\|_{\dot B^{\sigma}_{q,r}} =\|\mathcal{L}((-\Delta)^{\frac\sigma 2}u)\|_{ \dot B^{\sigma}_{q,r}}\]
so that
\[ \int_0^{+\infty} \|\mathcal{L}(u)(t,.)\|_{\dot B^{\sigma}_{q,1}} \dt \leq C_q  \int_0^{+\infty} \|u(t,.)\|_{ \dot B^{\sigma}_{q,r}} \dt \] and 
\[ \sup_{0<t<+\infty}  \|\mathcal{L}(u)(t.)\|_{\dot B^{\sigma}_{q,\infty}}  \leq C_q \sup_{0<t<+\infty}  \|u(t,.)\|_{\dot B^{\sigma}_{q,\infty}} . \] 

\subsection*{Product laws for $\dot B^{\frac 1 2}_{2,1}$} We are interested in multiplying a function $w\in \dot B^{\frac 1 2}_{2,1}$ by a function $u$ in $\dot B^{\frac 3 2}_{2,1}$ or in $L^{3/2,\infty}$. Multiplication by $w$ has for effect of loosing one derivative on $u$.

\noindent a) We have \[ \|uw\|_{\dot B^{1/2}_{2,1}}\leq C \|u\|_{\dot B^{3/2}_{2,1}}\|w\|_{\dot B^{1/2}_{2,1}}:\] just use paraproducts and write \[ (\sum_{k\leq j+3} \|\Delta_k u\|_\infty ) \|\Delta_j v\|_2\leq \|u\|_{\dot B^0_{\infty, 1}} \|\Delta_j v\|_2\leq C  \|u\|_{\dot B^{3/2}_{2, 1}} \|\Delta_j v\|_2\] and 
 \[ (\sum_{k\leq j-4} \|\Delta_k v\|_\infty ) \|\Delta_j u\|_2\leq  C2^{j} \|v\|_{\dot B^{-1}_{\infty, 1}} \|\Delta_j u\|_2\leq C2^j  \|v\|_{\dot B^{1/2}_{2, 1}} \|\Delta_j u\|_2\]
 
 \noindent b)  We have \[ \|uw\|_1\leq C \|u\|_{L^{3/2,\infty}}\|w\|_{\dot B^{1/2}_{2,1}}:\] just write $\dot B^{1/2}_{2,1}\subset L^{3,1}$.

\subsection*{Mild solutions in $L^1((0,+\infty), \dot B^{\frac 3 2}_{2,1})$}
If $u_0\in \dot B^{-\frac 1 2}_{2,1}$ then $e^{t\Delta}u_0\in \mathbb{Y}=L^1((0,+\infty), \dot B^{\frac 3 2}_{2,1})$. Indeed, \[ \|\Delta_j e^{t\Delta} u_0\|_2\leq C \min(1, \frac 1{t^22^{4j}}) \|\Delta_ju_0\|_2\] so that
\[ 2^{3j/2} \int_0^{+\infty} \|e^{t\Delta}\Delta_j u_0\|_2\dt\leq C2^{-j/2} \|\Delta_j u_0\|_2\int_0^{+\infty}  \min(2^{2j}t,\frac 1{2^{2j}t}) \frac{\dt}t =2C2^{-j/2} \|\Delta_j u_0\|_2.
\]:

\begin{proposition}\label{prop2} Let
\[B_\epsilon(u,v)=  \int_0^t e^{(t-s)\Delta}\Div \left(u \int_0^s   e^{\epsilon(s-\sigma)\Delta} \vN v\dsig \right) \ds,\] where $\epsilon\geq 0$. Then we have the inequality
\begin{equation} \|B_\epsilon(u,v)\|_{L^1((0,+\infty), \dot B^{\frac 3 2}_{2,1})}\leq C_0 \|u\|_{L^1((0,+\infty), \dot B^{\frac 3 2}_{2,1})} \|v\|_{L^1((0,+\infty), \dot B^{\frac 3 2}_{2,1})} \end{equation} where $C_0$ doesn't depend on $\epsilon$.
\end{proposition}

We now check that the bilinear operators $B_\epsilon$, $\epsilon\geq 0$, are equicontinuous on  $L^1((0,+\infty), \dot B^{\frac 3 2}_{2,1})$

\begin{proof} Let   $w_\epsilon=   \int_0^t   e^{\epsilon(t-s)\Delta} \vN v\ds  $. We have
    \[ \|w_\epsilon\|_{\dot B^{\frac 1 2}_{2,1}}\leq  \int_0^t   \| \vN v\|_{\dot B^{\frac 1 2}_{2,1}}\ds  \]
    so that
    \[ \|w_\epsilon\|_{L^\infty(\dot B^{\frac 1 2}_{2,1})}\leq   C \|  v\|_{L^1(\dot B^{\frac 3 2}_{2,1})} ,\]
    \[ \|u w_\epsilon\|_{L^1(\dot B^{\frac 1 2}_{2,1})}\leq   C \|  u\|_{L^1(\dot B^{\frac 3 2}_{2,1})} \|  v\|_{L^1(\dot B^{\frac 3 2}_{2,1})} \] and 
    \[ \|\frac 1 \Delta \Div(u w_\epsilon)\|_{L^1(\dot B^{\frac 3 2}_{2,1})}\leq   C \|  u\|_{L^1(\dot B^{\frac 3 2}_{2,1})} \|  v\|_{L^1(\dot B^{\frac 3 2}_{2,1})} .\] 
    As $B_\epsilon(u,v)=\int_0^{+\infty} e^{(t-s)\Delta} \Delta \left(\frac 1 \Delta \Div(u w_\epsilon)\right) \ds$, we conclude by maximal regularity that $B_\epsilon(u,v)$ is well controlled in $L^1((0,+\infty), \dot B^{3/2}_{2,1})$. 
    
    \subsection*{The case $u_0\in  L^{3/2,1}$}
   Let $u_0\in L^{3/2,1}$. As   $ L^{3/2,1}\subset \dot B^{-\frac 1 2}_{2,\infty}$, we already know that \[e^{t\Delta}u_0\in L^1((0,+\infty), \dot B^{\frac 3 2}_{2,1}).\] Moreover, as $L^{3/2,1}\subset L^{3/2,\infty}$, we know that \[e^{t\Delta}u_0\in L^\infty((0,+\infty), L^{3/2,\infty}).\] We may end the proof of Theorem \ref{theo1} by proving that the bilinear operators $B_\epsilon$, $\epsilon\geq 0$, are equicontinuous on  $L^1((0,+\infty), \dot B^{\frac 3 2}_{2,1})\cap L^\infty((0,+\infty), L^{3/2,\infty}). $ Let 
 \[  \|u\|_{\mathbb{Y}, 1}=\|u\|_{L^1((0,+\infty), \dot B^{\frac 3 2}_{2,1})} \text{ and }   \|u\|_{\mathbb{Y}, \infty}=\|u\|_{L^\infty((0,+\infty), L^{\frac 3 2,\infty})}.\] We already know that
 \[ \| B_\epsilon(u,v)\|_{\mathbb{Y},1}\leq C_0 \| u\|_{\mathbb{Y},1} \| v\|_{\mathbb{Y},1}.\] Moreover, we have
 \[ \|uw_\epsilon\|_1\leq C \|u\|_{L^{3/2,\infty}}\|w_\epsilon\|_{\dot B^{1/2}_{2,1}}\leq C' \|u\|_{L^{3/2,\infty}} \| v\|_{\mathbb{Y},1}.\] 
 We check that
 \[ \|\int_0^t e^{(t-s)\Delta} \Div f\ds\|_{L^{3/2,\infty}}\leq \|f\|_{L^\infty((0,+\infty), L^1)}:\]  we have \[ \|e^{(t-s)\Delta} \Div f \|_1\leq C \frac 1{(t-s)^{1/2}} \|f(s,.)\|_1\] and
  \[ \|e^{(t-s)\Delta} \Div f \|_3\leq C \frac 1{(t-s)^{3/2}} \|f(s,.)\|_1\] 
  so that, for every $A>0$, (integrating for $s<t-A$ and $s>t-A$)
$\int_0^t e^{(t-s)\Delta} \Div f\ds= I_A+J_A$ with \[\|I_A\|_1\leq C \sqrt A \|f\|_{L^\infty L^1}\text{ and }\|I_A\|_3\leq C \frac 1{ \sqrt A} \|f\|_{L^\infty L^1}.\] As it is true for every $A>0$, we find the control on  $   \|\int_0^t e^{(t-s)\Delta} \Div f\ds\|_{L^{3/2,\infty}}$. This gives 
 \[ \| B_\epsilon(u,v)\|_{\mathbb{Y},\infty}\leq C_0 \| u\|_{\mathbb{Y},\infty} \| v\|_{\mathbb{Y},1}. \tag*{\qedhere}\] 
 \end{proof}
 
 \noindent {\bf Remark: } We couldn't conclude directly by writing  $\int_0^t e^{(t-s)\Delta} \Div f\ds= \int_0^t e^{(t-s)\Delta}\Delta (\frac 1 \Delta \Div f)\ds$ with  $\frac 1 \Delta \Div f \in L^\infty L^{3/2,\infty}$, as the maximal regularity for the heat kernel is not true in $L^\infty L^{3/2,\infty}$. This is the same problem as for the Navier\ddh Stokes equations in $L^\infty L^{3,\infty}$ (Kozono and Nakao \cite{Kozo96}, Meyer \cite{Meye99}).
 
    \end{document}